\documentclass[11pt, twoside, leqno]{amsart}  
\usepackage{soul}
\usepackage{lipsum}
\usepackage{amsfonts}
\usepackage{graphicx}   
\usepackage{epstopdf}
\usepackage{float}

\usepackage{calligra}
\usepackage{amsfonts,amsmath,amsthm,amssymb}
\usepackage{mathtools}
\usepackage{hyperref}
\usepackage{autonum}
\usepackage{hhline}
\usepackage{array}
\usepackage[dvipsnames]{xcolor}
\usepackage{diagbox}
\usepackage{tcolorbox}
\usepackage{mdframed}
\usepackage{multicol}
\usepackage{graphicx}
\usepackage{subcaption}
\usepackage{upgreek}
\usepackage{moreverb}
\usepackage{bbm}
\usepackage[margin=1.34in]{geometry}
\usepackage{todonotes}
\usepackage{scalerel,amssymb}
\allowdisplaybreaks
\usepackage{mathrsfs}  
\usepackage{lineno}
\usepackage{todonotes}
\usepackage{tikz}
\usepackage[makeroom]{cancel}
\usepackage{appendix}
\usepackage{pgfplots}
\pgfplotsset{compat=1.18}
\usetikzlibrary{arrows.meta}
\usepackage[numbers,sort&compress]{natbib}
\definecolor{mygreen}{HTML}{43a047}
\usepackage{subcaption}
\usepackage{doi}
\usepackage{alphalph}
\usepackage{booktabs}
\usepackage{makecell}
\usepackage{adjustbox}
\usepackage{enumitem}
\makeatletter

\@namedef{subjclassname@2020}{\textup{2020} Mathematics Subject Classification}

\makeatother
\usepackage{array}
\newcolumntype{H}{>{\setbox0=\hbox\bgroup}c<{\egroup}@{}}

\newcommand{\Om}{\Omega}
\newcommand{\D}{\Delta}

\def \psit{\psi_t}
\def \psitt{\psi_{tt}}
\def \psittt{\psi_{ttt}}

\newcommand{\ddt}{\frac{\textup{d}}{\textup{d}t}}

\newcommand{\ds}{\, \textup{d} s }
\newcommand{\dx}{\, \textup{d} x}

\newcommand{\dxs}{\, \textup{d}x\textup{d}s}

\newcommand{\inttO}{\int_0^t \int_{\Omega}}
\newcommand{\intt}{\int_0^t}

\newcommand{\intO}{\int_{\Omega}}

\newcommand{\R}{\mathbb{R}} 
 
\newcommand{\Ltwo}{L^2(\Omega)}

\newtheorem{theorem}{Theorem}
\newtheorem{lemma}{Lemma}

\newtheorem{corollary}{Corollary}
\newtheorem*{assumption*}{Assumptions}
\newtheorem{assumption}{Assumption}

\newtheorem{remark}{Remark}
\numberwithin{lemma}{section}
\numberwithin{proposition}{section}
\numberwithin{theorem}{section}
\numberwithin{equation}{section}
\makeatletter
\newcommand{\leqnomode}{\tagsleft@true}
\newcommand{\reqnomode}{\tagsleft@false}
\makeatother

\newcommand{\genk}{\mathfrak{K}}
\newcommand\Lconv{\ast}

\definecolor{grey}{rgb}{0.5,0.5,0.5}
 
\usepackage{stackengine,scalerel}

\colorlet{brown}{brown!80!black}

\definecolor{darkgreen}{rgb}{0,0.5,0}

\newcommand{\alignedintertext}[1]{%
	\noalign{%
		\vskip\belowdisplayshortskip
		\vtop{\hsize=\linewidth#1\par
			\expandafter}%
		\expandafter\prevdepth\the\prevdepth
	}%
}

\title[]{ Well-posedness and global extensibility \\ criteria for time-fractionally damped Jordan--Moore--Gibson--Thompson equation}
  
     
\author[M. Meliani and B. Said-Houari]{Mostafa Meliani \and Belkacem Said-Houari}

\address{ 
	Department of Evolution Differential Equations \\ 
	Czech Academy of Sciences   \\ 
	\v Zitna 25,
	115 67 Praha 1, Czech Republic}
\email{meliani@math.cas.cz} 

\address{ 
	Department of Mathematics \\ 
	College of Sciences\\
	University of Sharjah
	\\
	P. O. Box: 27272, Sharjah, UAE}
\email{bhouari@sharjah.ac.ae} 
\begin{document}

\begin{abstract}
	In this paper, we consider the Jordan--Moore--Gibson--Thompson with a time-fractional damping term of the type $\delta   \textup{D}_t^{1-\alpha} \Delta \psit$ where we allow the challenging so-called critical case ($\delta=0$). This equation arises in the context of  acoustic propagation through thermally relaxed media.
  We tackle the question of long-time existence of the solution. More precisely, the goal of the paper is twofold: First,  we establish local well-posedness of the initial boundary value problem, where we also provide a lower bound on the final time of existence as a function of initial data. Second, we prove a regularity result which guarantees, under the hypothesis of boundedness of certain quantities, that the local solution can be extended to be global-in-time. 
\end{abstract}

\maketitle
\section{Introduction}
In this work, we consider the fractionally damped Jordan--Moore--Gibson--Thompson (fJMGT) equation:
\begin{equation}\label{eq:first_eq}
		\tau \psittt+(1+2k\psit) \psitt
	-c^2  \Delta \psi - \tau c^2  \Delta\psit
	- \delta   \textup{D}_t^{1-\alpha} \Delta \psit=0,
\end{equation}
on a bounded Lipschitz domain $\Omega\subset\R^d$ with $d\in \{1,2,3\}$. The parameters $\tau>0$, $\delta\geq 0$ (allowing the so-called critical case for the JMGT equation), while $k \in \R$. 
Here $\textup{D}_t^{\gamma}$ denotes the Caputo--Djrbashian \cite{Djrbashian:1966,Caputo:1967} 
 derivative of order $\gamma \in (0,1)$:
\begin{equation}
	\textup{D}_t^{\gamma} v = \frac{1}{\Gamma(1-\gamma)} \intt (t-s)^{-\gamma} v_t \ds \quad \text{for }\ v \in W^{1,1}(0,T), \, t \in (0,T). 
\end{equation}
 Hence, we can express the last term in \eqref{eq:first_eq} as
\begin{equation}
    \delta   \textup{D}_t^{1-\alpha} \Delta \psit= \delta \genk \Lconv\Delta \psitt=\delta\int_0^t g_\alpha(t-s)\Delta \psi_{tt}(s)\ds,
\end{equation}
where $g_\alpha$ is the Abel fractional integral kernel defined as 
\begin{equation}\label{eq:Abel_kernel}
	g_\alpha: t \mapsto \frac1{\Gamma(1-\alpha)} t^{-\alpha},\qquad \alpha \in (0,1). 
\end{equation}

In recent years, considerable attention has been devoted to the study of the JMGT equation and its linear counterpart, known as the MGT equation. See \cite{Kaltenbacher_2011,Klaten_2015,Trigg_et_al,P-SM-2019,kaltenbacher2019jordan,kaltenbacher2012well,Chen_Ikehata_2021,racke2021global} and references therein. 
Recently, nonlocal generalizations of these equations have been proposed in \cite{kaltenbacher2022time} using the Compte–Metzler fractional interpolations of the Fourier and Maxwell–Cattaneo flux laws \cite{Compte_Albert_Metzler}. In \cite{nikolic2024nonlinear} Nikoli\'{c}  studied the local well-posedness  and the singular limit $\tau \to 0^+$ of a nonlinear fractional JMGT equation in pressure form 
\[\tau \genk\Lconv u_{ttt} + u_{tt}  -\tau c^2 \genk\Lconv \Delta u_t -c^2  \Delta u - \delta  \D u_t = (ku^2)_{tt},\]
where the unknown is the pressure $u$ and $\ast$ denotes the Laplace convolution in time. 
 Here $\genk$ is the Abel kernel
\begin{equation}
     \genk(t)=\frac{1}{\Gamma(1-\alpha)}t^{-\alpha},\qquad \alpha\in (1/2,1),
\end{equation}
where $\Gamma(\cdot)$  denotes the Gamma function. 
Taking advantage of the specific structure of the equation ($\phi + \genk \Lconv \phi_t + \delta\D u_t = (ku^2)_{tt},$ with $\phi = u_{tt} - c^2 \D u$)
the author showed well-posedness of the homogeneous Dirichlet problem with initial data 
\[(u_0,u_1,u_2) \in H^2(\Omega)\cap H_0^1(\Omega) \times  H_0^1(\Omega) \times L^2(\Omega).\]

 Studies concerned with the well-posedness of the related linear Moore--Gibson--Thompson equation~\cite{meliani2023unified} and its long-time behavior have also been of interest.  In ~\cite{meliani2025energy}, the authors of this paper identified appropriate conditions on the convolution kernel that ensure the existence of solutions of the problem 
 \begin{equation}\label{Eq_Frac_2}
     \tau \genk\Lconv \psi_{ttt} + \psi_{tt}  -\gamma \genk\Lconv\Delta \psi_t -c^2 \Delta \psi - \delta  \D\tau \psi_t = 0,
 \end{equation}
 where the unknown is the velocity potential $\psi$. 
 Specifically, our result was based on the existence of regular resolvents and appropriate coercivity inequalities for the convolution kernel $\genk$. In addition, we showed in~\cite{meliani2025energy} the decay of the energy of the solution of \eqref{Eq_Frac_2},
 provided that the $\gamma>\tau c^2$, $\delta>0$ and $\genk$ is, among others, an exponentially decaying kernel. 
Equation \eqref{eq:first_eq} was first derived  in \cite{kaltenbacher2022time} as an equation for ultrasound wave propagation through media exhibiting anomalous diffusion, where heat conduction is described by  the following heat flux law:
\begin{equation}\label{eq:heat_flux_GFEIII}
	 (1+\tau \partial_t)\boldsymbol{q}(t) =\, -\kappa \textup{D}_t^{1-\alpha} \nabla \theta
\end{equation}
with $\alpha \in (0,1)$ and $\kappa$ is the thermal conductivity of the medium. Heat flux law \eqref{eq:heat_flux_GFEIII} was proposed as a generalization of the well-known Maxwell--Cattaneo heat flux law~\cite{cattaneo1958forme} 
\begin{equation}\label{Cattaneo}
    \tau\partial_t \boldsymbol{q}+\boldsymbol{q}=-\kappa \nabla \theta,
\end{equation}
which can be formally  obtained from \eqref{eq:heat_flux_GFEIII}  by setting $\alpha=1$. 
Note that the Maxwell--Cattaneo constitutive relation (\ref{Cattaneo}) can be
seen as a first-order approximation of a more general constitutive relation
(single$\--$phase$\--$lagging model; Tzou \cite{Tzou_1992}),
\begin{equation}
\boldsymbol{q}(x,t+\tau )=-\kappa \nabla \theta (x,t).  \label{Cattaneo_2}
\end{equation}
which  states that the temperature gradient
established at a point $x$ at time $t$ gives rise to a heat flux vector at $x $ at a later time $t+\tau $. The delay time $\tau $ is interpreted
as the relaxation time due to the fast-transient effects of thermal inertia
(or small-scale effects of heat transport in time) and is called the
phase$\--$lag of the heat flux. 

The Maxwell--Cattaneo heat conduction is the simplest generalization of the classical Fourier law (i.e., \eqref{Cattaneo_2} with $\tau=0$)  of heat conduction, which, on one hand, removes the infinite speed propagation predicted by the Fourier law and, on the other hand, leads to hyperbolic problems that are mathematically more challenging to address.  For instance, the use of Maxwell--Cattaneo law instead of the classical Fourier law in the equations of compressible fluid dynamics leads to the Jordan--Moore--Gibson--Thompson (i.e., equation \eqref{eq:first_eq} with $\alpha=1$) which is of a hyperbolic nature due to the presence of the term $\tau\psi_{ttt}$  \cite{jordan2014second}. It is known that the JMGT equation is relatively difficult to study compared to its parabolic counterpart (the Westervelt equation).  

 In practice, the parameter $\tau$ in \eqref{eq:heat_flux_GFEIII} and \eqref{Cattaneo_2} is small compared to the other parameters of the JMGT model, it is then natural to study the singular limit when $\tau\rightarrow 0^+$; see, e.g., \cite{kaltenbacher2020vanishing,kaltenbacher2019jordan,bongarti2020vanishing,bongarti2020singular} for references on the limiting behavior of \eqref{Cattaneo_2} and the aforementioned \cite{nikolic2024nonlinear,meliani2023unified} for $\tau$-limits of fractional JMGT equations. Of particular interest to the current work is the result in \cite{kaltenbacher2024vanishing}, where Kaltenbacher and Nikoli\'c studied the limiting behavior of $\tau\to0^+$ for a large family of JMGT equations with different assumptions on the kernel $\genk$. The study of \eqref{eq:first_eq} is coverd by \cite[Theorems 3.1]{kaltenbacher2024vanishing} with a restriction on the derivative order $\alpha>1/2$. \footnote{it is likely that this condition can be removed if one is not interested in $\tau\to0^+$ in~\cite{kaltenbacher2024vanishing}} 
 
The MGT and the JMGT equations with a memory term have also been investigated recently.
In \cite{lasiecka2017global}  Lasiecka considered the problem 
\begin{eqnarray}
    \tau u_{ttt} + u_{tt}  -\tau c^2 \Delta u_t -c^2 \Delta u - \delta \genk\Lconv \D(\tau u_t+ u) = (ku^2)_{tt},
\end{eqnarray}
where $\genk$ is assumed to be an exponentially decaying, integrable function. This assumption on the kernel allows the author to extract damping from the nonlocal term by adopting some techniques from the study of the viscoelastic wave equation  along the lines of~\cite{dafermos1970asymptotic}. It is important to mention that the method used in \cite{lasiecka2017global} cannot be applied to our problem \eqref{eq:first_eq} since the fractional-derivative convolution kernel given in \eqref{eq:Abel_kernel} does not satisfy the assumptions in \cite{lasiecka2017global}.  For more results about the  MGT with memory, the reader is referred to \cite{Bounadja_Said_2019,Liuetal._2019,dell2016moore} and to \cite{lasiecka2017global,NIKOLIC2021172} for the JMGT with memory. 

\subsection{Main Contributions} 
The main contributions of this paper, as outlined below, are twofold: First, establishing  a renewed local existence theory for equation \eqref{eq:first_eq} and second, proving a conditional global existence result by showing that under appropriate assumptions, the solution does not blow up in finite time.  

\subsection{Local well-posedness}
 The local well-posedness of \eqref{eq:first_eq} with homogeneous Dirichlet conditions and  initial data satisfying 
\[(\psi_0,\psi_1,\psi_2) \in \big(H^2(\Omega)\cap H_0^1(\Omega)\big) \times \big(H^2(\Omega)\cap H_0^1(\Omega)\big) \times H_0^1(\Omega)\]
has been the subject of the study in  \cite{kaltenbacher2022time}. The authors showed that the problem is well-posed using a fixed-point argument in which the final time of existence and the size of initial data are chosen small enough to obtain the self-mapping and the strict contractivity properties of the fixed-point mapping. From the proof of \cite[Theorem 4.1]{kaltenbacher2022time}, it seems not possible to obtain the existence of a solution beyond a certain time, even for arbitrarily small initial data due to the proof of strict contractivity where the final time appears in the estimate. Thus, proving strict contractivity in their proof requires reducing the size of the time interval irrespective of the data size.

The local well-posedness of \eqref{eq:first_eq}, was also covered, as mentioned previously, by \cite{kaltenbacher2024vanishing}. Although the fixed-point argument in \cite{kaltenbacher2024vanishing} allows for arbitrary final time and only relies on the smallness of initial data, 
it does not provide an estimate on the evolution of the final time of existence as a function of the size of the initial data.
The proofs in the aforementioned studies rely on a suitable linearization of the nonlinear equation together with a fixed-point argument for sufficiently small initial data and final time. 

In this work, we adopt a different approach. We do not linearize \eqref{eq:first_eq}  at the stage of the energy arguments but rather use results on the existence theory of nonlinear Volterra equations (e.g., \cite[Ch. 12]{gripenberg1990volterra}) to show that the semi-discrete Galerkin approximation is solvable up to a final time which may depend, \emph{a priori}, on the Galerkin approximation level. Using energy estimates and Gronwall's lemma, we show that the final time of existence is independent of the Galerkin approximation and give an explicit lower bound on the time of existence. 
This result is stated in Theorem~\ref{prop::local_well_posedness}. We point out that by avoiding the use of a fixed-point argument, we have better control on the estimates. Indeed, linearizing the equation, introduces right-hand side terms which have to be accounted for properly in the fixed-point argument. By keeping the nonlinear terms, we have a better handle on the growth of the energy of the solution, which has two advantages: First, knowing the rate of growth of the energy, we can more easily predict an estimate on the lower bound of the energy blow-up time. Second, it allows us to identify terms that contribute to the superlinear growth of the energy (which is the problematic part). Having identified these terms, we can state a conditional global existence result.

 It is worth mentioning that estimating the final time of existence through Gronwall's lemma in the context of second-order-in-time nonlinear wave equations has already been done in, e.g., \cite{dekkers2017cauchy} and \cite{john1990nonlinear}.

\subsection{Conditional global existence}
In the so-called \emph{critical} case $\delta=0$, the question of global  existence of solution and blow-up has remained unsolved for the nonlinear Jordan--Moore--Gibson--Thompson equation~\cite{jordan2014second}
\begin{equation}\label{eq:JMGT}
	\tau \psittt+ \psitt
	-c^2  \Delta \psi - \tau c^2  \Delta\psit= N_\psi,
\end{equation}
with $N_\psi = -2k\psit\psitt$ being the nonlinear term of Westervelt type.
 
We mention here the works of \cite{chen2019nonexistence,chen2021blow} which obtained blow-up for the Cauchy problem of the JMGT equation \eqref{eq:JMGT} with modified nonlinear terms $N_\psi = |\psi|^p$ and $|\psi_t|^p$ together with suitable restrictions on the exponent $p$. 
The type of the nonlinear terms in~\cite{chen2019nonexistence,chen2021blow} is important and is related to a family of work on blow-up of semi-linear evolution equations. In particular, it is related to the Strauss~\cite{strauss1981nonlinear} and Glassey exponents. We mention here also the related Fujita exponent~\cite{fujita1966blowing} for nonlinear wave equations; see, e.g.,~\cite{TODOROVA2001464}.

The work of \cite{chen2022influence} proved blow-up of the solution for \eqref{eq:JMGT} with Kuznetsov-type nonlinearity i.e., $N_\psi = 2k\psit\psitt + \nabla\psi_t\cdot\nabla\psi$. However, as pointed out by Nikoli\'c and Winkler in their recent work~\cite{nikolic2024infty}, there is a poor understanding of the qualitative behavior near blow-up for the different nonlinear JMGT equations. The work of \cite{nikolic2024infty} set to fill in this gap by showing that blow-up of \eqref{eq:JMGT} with the Westervelt-type nonlinearity is driven by inflation of the $L^\infty$-norm of the solution $\|\psi_t\|_{L^\infty(\Omega)}$ (as opposed to a gradient blow-up). 

In the nonlocal-in-time case, to the best of our knowledge, there are no results on blow-up of solutions or global existence for \eqref{eq:first_eq}. We set out to prove a conditional global existence result in 2D similar to that of \cite{nikolic2024infty}. We show that one can extend regular solutions in time as long as the quantity:
\[\|\psi_t\|_{H^1(\Omega)} + \|\psi_{tt}\|_{L^2(\Omega)}\]
remains bounded uniformly with respect to time.
The main ingredient of our proof is the Brezis--Galou\"et inequality, which was originally established to show the global well-posedness of the 2D cubic nonlinear Schr\"odinger equation~\cite{brezis1980nonlinear}; see Lemma~\ref{Lemma::brezis_gallouet}. The inequality is used to control the nonlinear term $(2k\psit\psitt)$ in \eqref{eq:first_eq}. Compared to the well-known embedding $H^2(\Omega)\hookrightarrow L^\infty(\Omega)$, the Brezis--Gallou\"et inequality shows that we only need to control the $H^2(\Omega)$ norm in a logarithmic way in the two-dimensional case. It is proved that this logarithmic growth does not result in finite time blow-up,  provided the aforementioned boundedness assumption is satisfied. This result is given in Theorem~\ref{thm:Beale_Kato_Majda_criterion}. 

We point out that, importantly, we are in the so-called critical case in the setting of equation~\ref{eq:first_eq}. Indeed by only relying on  Assumption~\ref{assumption1} below and allowing $\delta=0$, we do not exploit the potential regularizing properties resulting from the damping term. Our results hold then for a larger set of memory kernels. Furthermore, this approach allows us to simultaneously draw conclusions on the critical local-in-time JMGT equation ($\tau>0,\,\delta=0$).

\subsection{Organization of the paper}
The rest of the paper is organized as follows: In Section~\ref{Section_Prel}, we introduce the notation used in this work alongside useful lemmas and an assumption on the memory kernel $\genk$. Section~\ref{Section_local_wellp} is devoted to establishing a local well-posedness together with an explicit lower bound on the final time of existence. In Sections~\ref{Section_Cond_reg} and \ref{sec:cond_3D}, we obtain conditional regularity results to extend the final time of existence of solutions. First in the two-dimensional case, we then generalize the result to 3D in Section~\ref{sec:cond_3D},  where we rely on a dimension-dependent generalization of the Brezis--Gallou\"et inequality.

\section{Preliminaries}\label{Section_Prel}
\subsection{Notation}
Throughout the paper, we denote by $T>0$ the final propagation time. The letter $C$ denotes a generic positive constant
that does not depend on time, and can have different values on different occasions.  
We often write $f \lesssim g$ (resp. $f \gtrsim g$) when there exists a constant $C>0$, independent of parameters of interest such that $f\leq C g$ (resp. $f \geq C g$). 
We often omit the spatial and temporal domain when writing norms; for example, $\|\cdot\|_{L^p L^q}$ denotes the norm in $L^p(0,T; L^q(\Omega))$.
\subsection{Inequalities and embedding results} In the upcoming analysis, we shall employ the continuous embeddings $H^1(\Omega) \hookrightarrow L^4(\Omega)$ and $H^2(\Omega) \hookrightarrow L^\infty(\Omega)$ which are valid for $\Omega$ a bounded domain of $\R^d$ ($d\in\{1,2,3\}$).

We will also make a repeated use of  Young's $\varepsilon$-inequality 
\begin{equation}\label{Young_Inequality}
xy \leq \varepsilon x^n+C(\varepsilon) y^m, \quad \text{where}\quad \ x, y >0, \quad 1 <m,n <\infty,\quad \frac{1}{m}+\frac{1}{n}=1,
\end{equation}
and $C(\varepsilon)=(\varepsilon n)^{-m/n}m^{-1}$. In particular, we use the following two particular cases:
\begin{equation}
xy\leq \varepsilon x^2+\frac{1}{4\varepsilon} y^2
\end{equation}
and 
\begin{equation}\label{Young_2}
xy\leq \frac{1}{4} x^{4}+\frac{3}{4}y^{4/3}. 
\end{equation}
In the proof of Theorem~\ref{thm:Beale_Kato_Majda_criterion}, we will use two main ingredients. The first ingredient is the Brezis-Gallou\"et inequality in 2D and the second ingredient is the Ladyzhenskaya interpolation inequality. We really these inequalities below. 
\begin{lemma}[Lemma 1 in \cite{brezis1980nonlinear}]\label{Lemma::brezis_gallouet}
	Let $u \in H^2(\Omega)$ with $\Omega$ a domain in $\R^2$, then
	\begin{equation}
		\|u\|_{L^\infty(\Omega)} \leq C (\|u\|_{H^1(\Omega)}\sqrt{\ln(1+\|u\|_{H^2(\Omega)})}+1),
	\end{equation}
where $C$ is a constant that depends only on $\Omega$.
\end{lemma} 

\begin{proof}
	The inequality in \cite{brezis1980nonlinear} is proven in the case $\|u\|_{H^1(\Omega)} \leq 1$ but a straightforward update of \cite[Ineq. (3)]{brezis1980nonlinear} yields the desired result.
\end{proof}
Note that in 1D, we can make direct use of the Sobolev embedding $H^1(\Omega) \hookrightarrow L^\infty(\Omega)$.\\

We will furthermore need to use the following special case of Gagliardo--Nirenberg inequalities, first pointed out by Ladyzhenskaya~\cite{ladyzhenskaia1959solution}.
\begin{lemma}\label{lemma:ladyzhenskaya_inequality}
 Let $u \in H^1(\Omega)$  with $\Omega$ a bounded domain in $\R^d$ ($d=2,3$). Then 
\begin{equation}\label{ladyzhenskaya_ineq}
 	\|u\|_{L^4(\Omega)} \leq C \|u\|^{1-d/4}_{L^2(\Omega)} \| u\|_{H^1(\Omega)}^{d/4} 
 \end{equation}
where $C$ is a constant that depends only on $\Omega$.
\end{lemma}

\subsection{Generalization of the model}
To broaden the scope of our analysis, we can rewrite \eqref{eq:first_eq} as an integro-differential equation with an appropriate kernel as follows 
\begin{equation}\label{eq:first_eq_genk}
		\tau \psittt+(1+2k\psit) \psitt
-c^2 \Delta \psi - \tau c^2 \Delta\psit
- \delta \genk \Lconv\Delta \psitt=0.
\end{equation}

Concerning the kernel $\genk$ in \eqref{eq:first_eq_genk}, we make the following assumption
\begin{assumption}\label{assumption1}
	Let $T>0$. For all $y\in L^2(0, T; \Ltwo)$, it holds that for all $t\in(0,T)$
\begin{equation} 
	\begin{aligned}
		\int_0^{t} \intO \left(\genk\Lconv y \right)(s) \,y(s)\dxs\geq 	C_{\genk,T} 
		\int_0^{t} \|(\genk* y)(s)\|^2_{\Ltwo} \ds, 
	\end{aligned}
\end{equation}
where the constant $C_{\genk,T}>0$.
\end{assumption}
 Assumption \ref{assumption1} enforces the coercivity property of the memory term, and it is crucial in our analysis since it helps to extract a dissipative term that is necessary to control the norm of $\genk\Lconv \psi_{tt}$. This is used in two ways, to   bootstrap and obtain an estimate on $\psi_{ttt}$ (see \eqref{eq:partial_psittt} below) and then to pass to the limit in the weak form in Step (iv) of the proof of  Theorem~\ref{prop::local_well_posedness}. This assumption is standard when studying fractional PDEs. We refer the reader to, e.g., \cite[Section 5.2]{kaltenbacher2024limiting} for a discussion on how to verify this assumption.  In particular, it is verified for the Abel kernel~\eqref{eq:Abel_kernel}.

\section{Local-in-time well-posedness}\label{Section_local_wellp}
 We begin here by considering the local well-posedness of fJMGT 
 with Westervelt-type nonlinearity with a general kernel $\genk$ that satisfies 
 Assumption~\ref{assumption1}. Note that the local well-posedness of the equation studied here has been considered in the case $\genk = g_\alpha$  (with  $g_\alpha$ as in \eqref{eq:Abel_kernel})
 in \cite{kaltenbacher2022time}. Our goal here is to give an estimate of the lower bound of the final time of existence using Gronwall's type lemma. Our setting allows us to cover the inviscid (also known as critical) case of the Jordan--Moore--Gibson--Thompson {\color{violet} by setting $\genk=0$}.
 	
Let $T>0$. Consider the initial-boundary value problem 
\begin{equation}\label{eq:nonlinear_nonuniform}
	\left\{
	\begin{array}{ll}
		\tau \psittt+(1+2k\psit)\psitt
		-c^2 \Delta \psi - \tau c^2 \Delta\psit
		- \delta   \genk\Lconv \Delta \psitt = 0\qquad & \textrm{on } \Om\times(0,T), \\
		\psi = 0 & \textrm{on } \partial\Om\times(0,T),\\
		(\psi,\psit,\psitt)(0)= (\psi_0,\psi_1,\psi_2)&\textrm{on } \Om.
	\end{array} \right.
\end{equation}
 Let us introduce the space of solutions:
\begin{equation}\label{def:space_of_solutions}
	\begin{aligned}
	\mathcal{U}(0,T) = \Bigg\{u \in H^3(0,T; L^2(\Omega)) \cap W^{2,\infty}(0,T;H_0^1(\Om))& \cap   W^{1,\infty}(0,T;H^2(\Om)\cap H_0^1(\Om)) \Big|  \\
	&\sqrt\delta \genk \Lconv \D u_{tt} \in L^2(0,T;L^2(\Om))\Bigg\}.
	\end{aligned}
\end{equation}
We also define the following energy term 
\begin{equation}
\mathbf{E}[\psi](t):=\|\nabla\psitt(t)\|^2_{L^2(\Omega)}+  \|\D\psi(t)\|^2_{L^2(\Omega)}+  \|\D\psit(t)\|^2_{L^2(\Omega)}
\end{equation}
and its associated dissipation rate:
\begin{equation}
\mathbf{D}[\psi](t):=\intt \|\psittt\|_{L^2(\Om)}^2\ds + \delta \intt \| \genk \Lconv \D\psitt\|^2_{L^2(\Om)} \ds. 
\end{equation}
Hence, we have the following theorem.
\begin{theorem}\label{prop::local_well_posedness}  Let $\tau>0$, $\delta \geq 0$, and Assumption~\ref{assumption1} on $\genk$ hold. Assume that  $(\psi_0, \psi_1,\psi_2) \in H^2(\Om)\cap H_0^1(\Om)\times H^2(\Om)\cap H_0^1(\Om) \times H_0^1(\Om),$
	such that 
	\begin{equation}\label{N_0}
\|\psi_0\|_{H^2(\Omega)}^2 + \|\psi_1\|_{H^2(\Omega)}^2 + \|\psi_2\|_{H^1(\Omega)}^2 \leq N_0,
\end{equation}
	for some $N_0>0$. Then, there exits $T_*>0$ that only depends on $N_0$ such that the system \eqref{eq:nonlinear_nonuniform} admits a unique solution $\psi \in \mathcal{U}(0,t)$, at least up to $T_*$. Furthermore, for all $t \in [0,T_*]$, the following estimate holds:
 \begin{equation}
\mathbf{E}[\psi](t)+\mathbf{D}[\psi](t)\leq C(T_*)\mathbf{E}[\psi](0),
\end{equation}
where the constant $C(T_*)$ depends on $T_*$. 
\end{theorem}

\begin{proof}
	We prove Theorem \ref{prop::local_well_posedness} using the Galerkin approximation and the energy method. The proof will follow five steps:
	\begin{itemize}
		\item[(i)] We construct a finite-dimensional approximate solutions using the  Galerkin method.
		\item[(ii)] We establish, on these approximate solutions uniform \emph{a priori} estimates;
	\item[(iii)] Exploiting the \emph{a priori} estimates, we give a lower bound on the final time of existence;
		\item[(iv)] We pass to the limit, including in the nonlinear terms, thanks to the compactness properties of the  space of solutions;
		\item[(v)] We show that the constructed solution is the unique solution of \eqref{eq:nonlinear_nonuniform} in $\mathcal{U}$.
	\end{itemize}
	
\subsection*{Step (i): Constructing approximate solutions}
The proof follows by a Galerkin procedure in which we construct approximate solutions of \eqref{eq:nonlinear_nonuniform} of the form:
$$
\psi^n(x,t)  = \sum_{k=1}^{n} \xi_i^n(t) v_i(x),
$$
for $n\geq1$, where $\{v_i\}_{i=1}^\infty$ are smooth eigenfunctions of the Dirichlet-Laplace operator.

Let $V_n = \operatorname{span}\{v_1,\ldots,v_n\}$.  The semi-discrete problem is given by:
\begin{subequations}
\begin{equation}\label{eq:semi_discrete_weak_form}
		\big(\tau \psittt^n+(1+2k_1\psit^n)\psitt^n
		-c^2 \Delta \psi^n - \tau c^2 \Delta\psit^n
		- \delta   \genk\Lconv \Delta \psitt^n ,v_j\big) = 0,
\end{equation}
for all $j\in\{1,\ldots,n\}$, with approximate initial conditions
\begin{eqnarray}
(\psi^n,\psit^n,\psitt^n)(0)= (\psi_0^n,\psi_1^n,\psi_2^n),
\end{eqnarray}
\end{subequations}
 taken as the $L^2(\Om)$ projection of $(\psi_0,\psi_1,\psi_2)$ onto $V_n$. Note that with this choice of approximate initial data, the following stability estimates hold:
 
\begin{equation}\label{eq:stability_initial_data}
	\|\D \psi_0^n\|_{L^2(\Omega)} \leq \|\D \psi_0\|_{L^2(\Omega)}, \quad \|\D \psi_1^n\|_{L^2(\Omega)} \leq \|\D \psi_1\|_{L^2(\Omega)}, \quad \|\nabla \psi_2^n\|_{L^2(\Omega)} \leq \|\nabla \psi_2\|_{L^2(\Omega)}
\end{equation}
for all $n\geq 1$.
This can be seen by writing down the definition of the $L^2(\Om)$ projection:
$$(\psi_0^n,\phi^n) = (\psi_0,\phi^n), \quad \textrm{ for all } \phi^n \in V_n.$$
Since $V_n$ is spanned by smooth eigenfunctions of the Dirichlet Laplacian, $\phi^n = \D^2 \psi_0$ is a valid test function. Integrating by parts and using Cauchy-Schwartz inequality, yields the desired estimate. The other estimates  are shown to hold in a similar fashion.	

The resulting is a system of Volterra equations of the second kind for the vector quantity 
$$\chi = [{\xi_1^n},\ldots,{\xi_n^n},{\xi_1^n}_{t},\ldots,{\xi_n^n}_{t}, {\xi_1^n}_{tt},\ldots,{\xi_n^n}_{tt} ].$$ 
In particular the equation for $\chi : \R^+ \mapsto \R^{3n}$ reads as:
\begin{equation}\label{eq:chi_equation}
	\chi_t = H(\chi) (t), \qquad \qquad \chi(0) = \chi_0,
\end{equation}
where $\chi_0$ is the concatenation of the coefficient of $(\psi_0^n,\psi_1^n,\psi_2^n)$ in $V_n$. The function $H(\chi)(\cdot)$ is inferred from the relationship between the first $2n$ components of $\chi_t$ and the last $2n$ of $\chi$. For the last $n$ components, the form of $H(\chi)(\cdot)$ is inferred from \eqref{eq:semi_discrete_weak_form}.
 In other words, the function $H(\chi)(\cdot) = \textup{L} \chi + \textup{NL}(\chi)$ is the sum of a linear part 
	\[\textup{L} \chi = \begin{bmatrix}
		0_n & I_n & 0_n \\ 
		0_n & 0_n & I_n \\
		-\frac{c^2}\tau K & c^2 K  & -\frac1\tau M  \\
		\end{bmatrix} \chi + \frac\delta\tau \begin{bmatrix}
		0_n & 0_n & 0_n \\ 
		0_n & 0_n & 0_n \\
		0_n & 0_n &  K   \\
	\end{bmatrix} \genk \Lconv \chi , \]
	where $0_n$ and $I_n$ are the $n\times n$ null and identity matrices respectively,  
	and a nonlinear (quadratic) part 
	\[\textup{NL}(\chi)  = \frac1\tau \chi^T\begin{bmatrix}
		0_n & 0_n & 0_n \\ 
		0_n & 0_n & 0_n \\
		 0_n& -2k_1M & 0_n \\
	\end{bmatrix} \chi . \] 
Above $M$ and $K$ are the mass and stiffness matrix defined by:
\[ M_{i,j} = \int_\Omega v_i \ v_j \dx, \qquad  K_{i,j} = \int_\Omega \nabla v_i \cdot \nabla v_j \dx, \]
for $1 \leq i,j \leq n$.

Existence theory of systems of nonlinear Volterra equations yields the existence of a unique solution $\chi \in W^{1,1}(0,t_n;V_n)$ up to time $t_n$; see \cite[Chapter 12, Theorem 3.1]{gripenberg1990volterra}. The \emph{a priori} estimates will show that $t_n \geq T_*(N_0)$.

\subsection*{Step (ii): Establishing uniform estimates}
We establish here \emph{a priori} estimates of the solutions that are uniform with respect to the approximation level $n$.
In what follows, we omit the superscript $n$ to simplify notation. Furthermore, let us introduce the following energy 
\begin{equation}\label{eq:Energy_def}
	E(t) :=  \frac\tau2 \|\nabla\psitt(t)\|^2_{L^2(\Omega)}+\frac{\tau c^2}2  \|\D\psit(t)\|^2_{L^2(\Omega)}.
\end{equation}

Let $T>0$ be an arbitrary final time, which will be refined in Step (iii) of this proof to establish a lower bound on the final time of existence. 
Let $t\in[0,T]$. We begin by testing \eqref{eq:nonlinear_nonuniform}  with $-\D \psitt$ and integrating over $(0,t)$, to obtain 
\begin{equation}\label{eq:first_estimate_linearized}
	\begin{aligned}
		E(t)+ \intt \|\nabla \psitt\|_{L^2(\Om)}^2\ds + \intt \delta(\genk\Lconv \D\psitt, \D\psitt) \ds \\
		=E(0) - \intt   (c^2\D\psi - 2k_1\psit \psitt,\D\psitt)_{L^2(\Omega)}\ds.
	\end{aligned}
\end{equation}

We will show in what follows how to control each of the arising terms in \eqref{eq:first_estimate_linearized}. First, using Assumption~\ref{assumption1} on the kernel, it holds  that 
\begin{equation}\label{Est_Conv}
\intt (\genk\Lconv \D\psitt, \D\psitt) \ds \geq C_{\genk,T}   \intt \| \genk \Lconv \D\psitt\|^2_{L^2(\Om)} \ds.
\end{equation}
Furthermore, using integration by parts, we get
\begin{equation}
	\intt  (\psit \psitt, \D\psitt)_{L^2(\Omega)} \ds =  - \intt  (\nabla\psit \psitt,\nabla\psitt)_{L^2(\Omega)} \ds - \intt  (\psit \nabla\psitt,\nabla\psitt)_{L^2(\Omega)} \ds.
\end{equation}
By H\"older's inequality, we infer that 
\begin{equation}\label{eq:psit_psitt_product}
	\begin{aligned}
	&\left| \intt  (\nabla\psit \psitt,\nabla\psitt)_{L^2(\Omega)} \ds + \intt  (\psit \nabla\psitt,\nabla\psitt)_{L^2(\Omega)} \ds\right| \\
 &\leq  \intt \|\nabla \psit \|_{L^4(\Om)} \| \psitt \|_{L^4(\Om)} \|\nabla \psitt \|_{L^2(\Om)}\ds + \intt \|\psit\|_{L^{\infty}(\Omega)} \|\nabla\psitt\|^2_{L^2(\Om)}\ds.
	\end{aligned}
\end{equation}

Note that the approximate solutions $\psi$ belong to $H^2(\Om)\cap H^1_0(\Om)$.   Hence,  by standard Sobolev embeddings and elliptic regularity results, the following inequalities hold:
		 \begin{equation}\label{some_inequalities}
		 	\begin{aligned}
		 	\|\nabla \psit \|_{L^4(\Om)} &\lesssim \|\nabla \psit \|_{H^1(\Om)} \lesssim \|\D \psit \|_{L^2(\Om)},\\
		  \|\psitt \|_{L^4(\Om)} &\lesssim \|\psitt \|_{H^1(\Om)} \lesssim \|\nabla \psitt \|_{L^2(\Om)}, \\
 \|\psit\|_{L^{\infty}(\Omega)} &\lesssim \|\psit\|_{H^{2}(\Omega)}\lesssim \|\D\psit\|_{L^{2}(\Omega)}.
 \end{aligned}
\end{equation} 
 Using \eqref{some_inequalities} and Young's inequality with appropriate exponents, we can bound \eqref{eq:psit_psitt_product} as follows
\begin{equation}
	\begin{aligned}
		\Big|\intt  (\psit \psitt,\D\psitt)_{L^2(\Omega)} \ds\Big|&
		\lesssim   \intt E^{3/2}(s)\ds.
	\end{aligned}
\end{equation} 
On the other hand, we have
\begin{align}\label{Linear_Term_0}
	-\intt  (\D\psi,\D\psitt)_{L^2(\Omega)}\ds = \intt (\D\psit,\D\psit)_{L^2(\Omega)} \ds - (\D\psi,\D\psit)_{L^2(\Omega)}\big|_0^t.
\end{align}
The first term on the right-hand side of \eqref{Linear_Term_0} is estimated by
\begin{equation}\label{Linear_Term}
\intt (\D\psit,\D\psit)_{L^2(\Omega)} \ds \lesssim \int_0^t E(s)\ds.
\end{equation}
Additionally,
using the identity 
\begin{equation}
\Delta\psi(t)=\Delta\psi_0+\int_0^t\Delta\psi_t(s)\ds,
\end{equation}
together with the $\varepsilon$-Young inequality, 
 we can estimate the last term in \eqref{Linear_Term_0} as follows: 
\begin{align}
	(\D\psi,\D\psit)_{L^2(\Omega)}\big|_0^t \leq & \frac{1}{2\varepsilon} \left(t \intt \|\D \psi_t\|^2_{L^2(\Om)}\ds + \|\D \psi(0)\|^2_{L^2(\Om)}\right) \\&
	+ \varepsilon \|\D \psit\|_{L^2(\Om)}^2 + \frac12 \left(\|\D \psi (0)\|^2_{L^2(\Om)} + \|\D \psi_t (0)\|^2_{L^2(\Om)}\right),
\end{align}  
where $\varepsilon>0$ is an arbitrary constant. Inserting the above estimate into \eqref{eq:first_estimate_linearized} and selecting $\varepsilon$ small enough, we get  
\begin{equation}\label{eq:first_estimate_linearized_Main}
	\begin{aligned}
		 E(t)+ \intt \|\nabla \psitt\|_{L^2(\Om)}^2\ds
		 + \delta C_{\genk,T} \intt \| \genk \Lconv \D\psitt\|^2_{L^2(\Om)} \ds \\
		\leq  CN_0
		+ C(T) \intt E^{3/2}(s) \ds+\int_0^t E(s)\ds, 
	\end{aligned}
\end{equation}

where $N_0$ is the size of the initial data defined in \eqref{N_0}.

\subsection*{Step (iii): Lower bound on the final time of existence}
 It remains for us to show that inequality \eqref{eq:first_estimate_linearized_Main} implies that the energy $E$ remains bounded up to a $n$-independent final time $T_\textup{max}$ ($n$ being the Galerkin approximation level).  Recall that $T>0$  can be arbitrarily large. In what follows, we show that the lowest bound on the final time of existence is $\min\{T,T_0(N_0,T)\}$ where $T_0(N_0,T)$ depends on the size of the initial data and is defined in \eqref{def:T0N0} below. 
To achieve this goal, we use Gronwall's lemma; see, e.g., \cite[p. 47]{john1990nonlinear}. In particular, inequality~\eqref{eq:first_estimate_linearized_Main} shows that  $$E(t)\leq z(t),$$
where $z(t)$ is the solution of
	 \begin{equation}\label{eqn:bound_equation}
			z'=z(s)+ C(T)z(s)^\frac32,
		\end{equation}
		with $z(0) = z_0= CN_0$.
		By setting $u = - z^{-\frac12}$ ($z$ non-negative), we can write the equation for $u$ as
		  \begin{equation}
			 	2 u' + u = C(T),
			 \end{equation}
		with $u(0) = u_0= - z_0^{-\frac12} $, which can be solved by 
		$u :t \longmapsto u_0 e^{-t/2} - C(T)(e^{-t/2}-1)$.
		Then the solution to equation \eqref{eqn:bound_equation} can be expressed as:
			\begin{equation}   
					z = \frac{1}{(-{z_0^{-1/2}}e^{-t/2} +C(T )(1-e^{-t/2}))^2}.
	\end{equation}		
This implies that $E(t)$ remains bounded for $t \in (0,T_*)$ with $$T_* = \min\{T, T_0(N_0,T)\}>0$$ 
where 
\begin{equation}\label{def:T0N0}
	T_0(N_0,T)= 2 \log \left(\frac{z_0^{-1/2}+C(T)}{C(T)}\right).
\end{equation}
In fact, $T_*$ can be made independent of $T$; see Remark~\ref{remark:final_time} for a comment on how to find the optimal lowest bound of the final time of existence $T_*$ by changing the choice of $T$. Henceforth, we use the so-constructed final time of existence $T_*(N_0)$ in the estimates and Galerkin limits below.

\subsection*{Estimates on \texorpdfstring{$\D\psi$}{D psi} and \texorpdfstring{$\psittt$}{psi ttt}}
 Before passing to the limit, we extract two additional estimates for $\D\psi$ and $\psittt$. First, it is straightforward to obtain an estimate on $\psi$ by using the fact that 
 \begin{equation}
     \psi(t)=\psi_0+\int_0^t \psi_t(s)\ds. 
 \end{equation}
 Hence, this yields 
	\begin{equation} \label{eq:Dpsi_estimate}
		\|\D \psi(t)\|^2_{L^2(\Om)} \leq 2 \left(t^2 \sup_{0\leq s\leq t} \|\D \psi_t(s)\|^2_{L^2(\Om)} + \|\D \psi(0)\|^2_{L^2(\Om)}\right). 
	\end{equation}
Furthermore, by considering:
$$\tau \psittt= - (1+2k_1\psit)\psitt
+c^2 \Delta \psi + \tau c^2 \Delta\psit
+ \delta   \genk\Lconv \Delta \psitt,$$
we can estimate 
\begin{equation}\label{eq:partial_psittt}
	\begin{aligned}
	\|\psittt\|_{L^2(0,t;L^2(\Om))} \lesssim& (1+\|\psit\|_{L^\infty(0,t;L^\infty(\Om))})\|\psitt\|_{L^2(0,t; L^2(\Om))} \\&+ \|\D \psi\|_{L^2(0,t;L^2(\Om))} + \|\D \psit\|_{L^2(0,t;L^2(\Om))} + \|\genk \Lconv \D \psitt\|_{L^2(0,t;L^2(\Om))} \\
	\lesssim& (1+\|\D\psit\|_{L^\infty(0,t;L^2(\Om))})\|\nabla \psitt\|_{L^2(0,t;L^2(\Om))} \\&+ \|\D \psi\|_{L^2(0,t;L^2(\Om))} + \|\D \psit\|_{L^2(0,t;L^2(\Om))} + \|\genk \Lconv \D \psitt\|_{L^2(0,t;L^2(\Om))},
	\end{aligned}
\end{equation}
where  we have used the elliptic regularity estimate $\|\psit\|_{L^\infty(0,t;L^\infty(\Om))} \lesssim \|\D \psit\|_{L^\infty(0,t;L^2(\Om))}$ and the Poincar\'e--Friedrichs inequality $\|\psitt\|_{L^2(0,t; L^2(\Om))} \lesssim \|\nabla \psitt\|_{L^2(0,t;L^2(\Om))}$.

Combining \eqref{eq:first_estimate_linearized_Main}, \eqref{eq:Dpsi_estimate} and \eqref{eq:partial_psittt} yields the desired estimate:
\begin{equation}
\mathbf{E}[\psi](t)+\mathbf{D}[\psi](t)\leq C(T_*)\mathbf{E}[\psi](0). 
\end{equation}

\subsection*{Step (iv): Passing to the limit}
From the previous estimates, we know that the sequence of approximate solutions $\psi^n$ (where we stress again the dependence on the Faedo--Galerkin approximation level $n$) stays bounded uniformly with respect to $n$. In particular, we have the following weak(-$\star$) convergence of a subsequence (which we do not relabel):
 \begin{alignat}{2}
	\psi^{n} &\stackrel{}{\relbar\joinrel\rightharpoonup} \psi \quad &&\textrm{weakly-$*$ in } L^\infty(0,T_*; H^2(\Om)\cap H_0^1(\Om)),\\
	\psit^{n} &\stackrel{}{\relbar\joinrel\rightharpoonup} \psit \quad &&\textrm{weakly-$*$ in }L^\infty(0,T_*; H^2(\Om)\cap H_0^1(\Om)), \\
	\psitt^{n} &\stackrel{}{\relbar\joinrel\rightharpoonup} \psitt \quad &&\textrm{weakly-$*$ in }L^\infty(0,T_*; H_0^1(\Om)), \\
	\psittt^{n} &\stackrel{}{\relbar\joinrel\rightharpoonup} \psittt \quad &&\textrm{weakly \, \, in }L^2(0,T_*; L^2(\Om)).\\
\end{alignat}
Using the well-known Rellich--Kondrachov theorem and Aubin--Lions--Simon lemma~\cite[Thoerem 5]{simon1986compact}, we know that there exist a strongly convergent subsequence (again not relabeled) in the following sense
 \begin{alignat}{2}\label{eq:strong_convergence}
	\psi^{n} &\longrightarrow \psi \quad &&\textrm{strongly in } C([0,T_*]; H_0^1(\Om)),\\
	\psit^{n} &\longrightarrow \psit \quad &&\textrm{strongly in }C([0,T_*]; H_0^1(\Om)), \\
	\psitt^{n} &\longrightarrow \psitt \quad &&\textrm{strongly in }C([0,T_*]; L^2(\Om)). \\
\end{alignat}

Additionally, since $\sqrt{\delta}\genk \Lconv \D\psitt^n$ and $\psit^n\psitt^n$ are bounded in $L^2(0,T_*;L^2(\Om))$ and $L^\infty(0,T_*;L^2(\Om))$, we infer that 
 \begin{alignat}{2}
	\sqrt{\delta} \genk \Lconv \D\psitt^n &\stackrel{}{\relbar\joinrel\rightharpoonup} w \quad &&\textrm{weakly \,\, in } L^2(0,T_*; L^2(\Om)),\\
	\psit^n\psitt^n &\stackrel{}{\relbar\joinrel\rightharpoonup} y \quad &&\textrm{weakly-$*$ in }L^\infty(0,T_*; L^2(\Om)).
\end{alignat}
To show that $w= \sqrt{\delta} \genk \Lconv \D\psitt$ weakly, let $\phi\in C([0,T_*]; H^2(\Omega) \cap H^1_0(\Om))$ and $n\geq1$. Then, integration over time and space yields
\[ \inttO \sqrt{\delta}
\genk \Lconv \D\psitt^n \phi \dxs = \inttO \sqrt{\delta} \genk \Lconv \psitt^n \D\phi \dxs,\]
using the strong convergence of $\psitt^n\longrightarrow\psitt$ and passing to the limit in $n$ yields the desired result.
Similarly $y = \psit\psitt$ weakly using the fact that $\psit^n\psitt^n$ is the product of a weakly convergent sequence in $L^\infty(0,T_*; H^2(\Om)\cap H_0^1(\Om))\hookrightarrow L^\infty(0,T_*; L^\infty(\Om)$ and a strongly convergent sequence in $C([0,T_*]; L^2(\Om))$.
 
This is enough to show that the limit $\psi \in \mathcal{U}(0,T_*)$ solves the weak form 
\begin{equation}
	\intt\big(\tau \psittt+(1+2k_1\psit)\psitt
	-c^2 \Delta \psi - \tau c^2 \Delta\psit
	- \delta   \genk\Lconv \Delta \psitt, v\big) \ds = 0,
\end{equation}
for all $v \in L^2(0,T_*;L^2(\Om))$.
Attainment of initial data of the constructed solution is guaranteed by the combination of the strong convergence of $\psi,\, \psit, \psitt$ and the convergence of the $L^2(\Om)$ projections of the initial data.

\subsection*{Step (v): Uniqueness of solutions} 
Let $\psi^1$, $\psi^2 \in \mathcal{U}(0,T_*)$  be two solutions of \eqref{eq:nonlinear_nonuniform} and denote their difference by $\bar\psi = \psi^1-\psi^2$. Then $\bar \psi \in \mathcal{U}(0,T_*)$ must solve 
\begin{equation}
	\intt\big(\tau \bar\psi_{ttt}+(1+2k_1\psit^1)\bar\psi_{tt}
	-c^2 \Delta \bar\psi - \tau c^2 \Delta\bar\psi_{t}
	- \delta   \genk\Lconv \Delta \bar\psi_{tt}, v\big) \ds= \intt(- 2k_1 \bar\psi_t \psitt^2,v)\ds,
\end{equation}
for all $v \in L^2(0,T;L^2(\Om))$, with $(\bar\psi, \bar\psi_t, \bar\psi_{tt})(0) = (0,0,0)$.

Note that $v = \bar\psi_{tt}$ is a valid test function, which yields after computation
\begin{equation}\label{psi_bar_Energy}
	\begin{aligned}
		\frac\tau2 \|\bar\psi_{tt}(t)\|^2_{L^2(\Omega)}+ \intt \|\bar\psi_{tt}\|_{L^2(\Om)}^2\ds
		+\frac{\tau c^2}2  \|\nabla \bar\psi_t(t)\|^2_{L^2(\Omega)} + \intt \delta(\genk\Lconv \nabla\bar\psi_{tt}, \nabla\bar\psi_{tt}) \ds \\
		= \intt   (c^2\D\bar\psi -  2k_1\psit^1 \bar\psi_{tt} - 2k_1 \bar\psi_t \psitt^2,\bar\psi_{tt})_{L^2(\Omega)}\ds.
	\end{aligned}
\end{equation}
First, we have by Assumption~\ref{assumption1}
\[\intt \delta(\genk\Lconv \nabla\bar\psi_{tt}, \nabla\bar\psi_{tt}) \ds \geq 0.\]

Similarly to the arguments employed in establishing the energy bounds in step (ii), we intend to show that the terms on the right-hand side can be controlled by the left-hand side. To this end, we integrate the first term by parts in both space and time, we obtain 
\begin{align}
	\intt  (\D\bar\psi,\bar\psi_{tt})_{L^2(\Omega)}\ds = \intt (\nabla\bar\psi_t,\nabla\bar\psi_t)_{L^2(\Omega)} \ds - (\nabla\bar\psi,\nabla\bar\psi_t)_{L^2(\Omega)}\big|_0^t.
\end{align}

Further we use the fact  that $\bar\psi(0)= \bar\psi_t(0) = 0$ to estimate
\begin{align}\label{epsilon_tesm}
(\nabla\bar{\psi},\nabla\bar{\psit})_{L^2(\Omega)}\big|_0^t \leq & \frac{1}{2\varepsilon} t \intt \|\nabla \bar{\psi}_t\|^2_{L^2(\Om)}\ds  
	+ \varepsilon \|\nabla \bar{\psit}\|_{L^2(\Om)}^2 
\end{align}
where $\varepsilon>0$ is chosen small enough so that the second term on the right-hand side of \eqref{epsilon_tesm} is absorbed by the left-hand side of \eqref{psi_bar_Energy}.

Further, we have
\begin{equation}
	\begin{aligned}
\intt   (\psit^1 \bar\psi_{tt} + \bar\psi_t \psitt^2,\bar\psi_{tt})_{L^2(\Omega)}\ds \leq &	\intt \|\psit^1\|_{L^\infty(\Omega)}\|\bar \psi_{tt}\|^2_{L^2(\Omega)} \ds\\ &+ \intt\|\psitt^2\|_{L^4(\Omega)} \|\bar{\psi}_t\|_{L^4(\Omega)}\|\bar\psi_{tt}\|_{L^2(\Omega)}  \ds \\ 
\leq &	\intt \|\psit^1\|_{H^2(\Omega)}\|\bar \psi_{tt}\|^2_{L^2(\Omega)} \ds\\ &+ \frac12 \intt\|\psitt^2\|_{H^1(\Omega)} \left(\|\nabla\bar{\psi}_t\|^2_{L^2(\Omega)} + \|\bar\psi_{tt}\|^2_{L^2(\Omega)} \right) \ds,
\end{aligned}
\end{equation}
where in the last line, we have made use of Sobolev embeddings and elliptic regularity results.

Since $\psi^1$, $\psi^2 \in \mathcal{U}$, then necessarily $t \mapsto \|\psitt^2(t)\|_{H^1(\Omega)}$ and $t \mapsto \|\psit^1(t)\|_{H^2(\Omega)}$ are functions in $L^1(0,T_*)$. This allows us to make use of the classical Gronwall's inequality~(see, e.g., \cite[Chapter XVIII, \S 5, Lemma 1]{dautray2000mathematicalVolume5}) to conclude that $\bar \psi = 0$. Thus ensuring uniqueness of solutions.
\end{proof}

\begin{remark}[Sharpness of the final time of existence estimate]\label{remark:final_time}
	In the course of the proof of Theorem~\ref{prop::local_well_posedness}, we provided an estimate for the final time of existence by means of picking $T>0$ and calculating $T_0=T_0(N_0,T)$ according to \eqref{def:T0N0}. The minimal final time of existence is then given as the $\min\{T, T_0(N_0,T)\}$. Graphically, the optimal choice of $T$ corresponds to the abscissa of the intersection of the red and blue lines in Figure~\ref{fig:plot_T}.  
		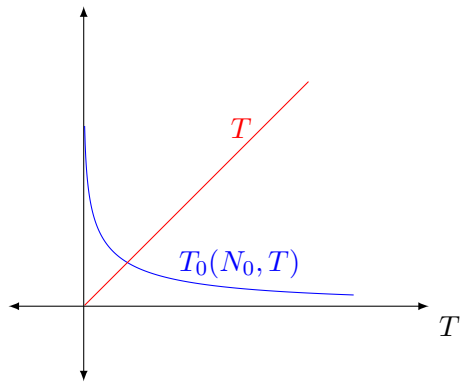
\begin{figure}[ht]
			\centering
		\begin{tikzpicture}[]
			\begin{axis}[width=3in,axis equal image,clip=false,
				axis lines=middle,
				xmin=-0,xmax=3,
				domain=-0:3, samples=201,
				xlabel=$T$,ylabel=$$,
				ymin=-0,ymax=2.5,
				restrict y to domain=-0:2.5,
				enlargelimits={abs=1cm},
				axis line style={latex-latex},
				ticklabel style={font=\tiny,fill=white},
				xtick={\empty},ytick={\empty},
				xlabel style={at={(ticklabel* cs:1)},anchor=north west},
				ylabel style={at={(ticklabel* cs:1)},anchor=south west}
				]
				\addplot[samples=501,domain=0.01:3,blue] {log10((x + 1)/x)} node[above,pos=0.7]{$T_0(N_0,T)$};
				\addplot[samples=400,domain=0.01:3, red] {x} node[above,pos=0.7]{$T$};
			\end{axis}
		\end{tikzpicture}
	\caption{Plot of $T$ (red) and $T_0(N_0,T)$ (blue) as a function of $T$. }
	\label{fig:plot_T}
	\end{figure}

By setting
 \[T_*(N_0) = \sup_{T>0} \min\{T, T_0(N_0,T)\},\] the estimate of the final time of existence only depends on the initial data size $N_0$ specified in \eqref{N_0}.
\end{remark}

\section{Global existence in 2D}\label{Section_Cond_reg}
We establish in this section a criterion for the extension of the solution beyond the critical time $T_*(N_0)$ identified in Theorem~\ref{prop::local_well_posedness} along the lines of, e.g., Beale--Kato--Majda argument for the 3-D Euler equation established in~\cite{beale1984remarks}. 
Theorem~\ref{thm:Beale_Kato_Majda_criterion} covers, like the rest of this work, the critical case $\delta = 0$.

In this section, we focus mainly on the 2-dimensional case. This is due to the availability of the crucial Brezis--Gallou\"et inequality of Lemma~\ref{Lemma::brezis_gallouet}. This inequality translates the idea that in 2D, the $H^1(\Omega)$ norms provides ``almost" control of the $L^\infty(\Omega)$ norm, i.e., we have the embedding $H^1(\Omega)\hookrightarrow L^p(\Omega)$ for all $1 \leq p<\infty$, such that the $H^2(\Omega)$-norm only needs to be accounted for in a logarithmic way. The one-dimensional case follows from the classical embedding $H^1(\Omega) \hookrightarrow L^\infty(\Omega)$.

\begin{theorem}\label{thm:Beale_Kato_Majda_criterion}
	Let $\Omega$ be a bounded Lipschitz domain in  $\R^d$, $d=1,2$ and let Assumption~\ref{assumption1} on $\genk$ hold. Assume that $\tau>0, \, \delta\geq 0$ and let $\psi$ be a solution to \eqref{eq:nonlinear_nonuniform} and suppose that there is a time $\bar T$ such that the solution cannot be continued in $\mathcal{U}$ to $T=\bar T$. Assume that $\bar T$ is the first such time. Then 
	\begin{equation}
		\limsup_{t\uparrow \bar T} \|\psi_t\|_{H^1(\Omega)} + \|\psi_{tt}\|_{L^2(\Omega)} = \infty.
	\end{equation}
\end{theorem}

We state a corollary, which is a direct consequence of the proof of this Theorem \ref{thm:Beale_Kato_Majda_criterion}.
\begin{corollary}\label{corr:Beale_Kato_Majda}
	Let $\Omega$ be a bounded Lipschitz domain in  $\R^d$, $d=1,2$ and let Assumption~\ref{assumption1} on $\genk$ hold. Assume that $\tau>0, \, \delta\geq 0$. For a solution of \eqref{eq:nonlinear_nonuniform}, suppose there are constants $M_0$ and $\bar T$ such that for all $T<\bar T$, we have the \emph{a priori} estimate:
 \begin{equation}\label{Blow_up_Cond}
     \sup_{t\in (0,T)}\|\psi_t(t)\|_{H^1(\Omega)}  + \|\psi_{tt}\|_{L^2(\Omega)} \leq M_0.
 \end{equation}
	Then the solution can be continued in $\mathcal{U}$ to the interval $[0,\bar T]$.
\end{corollary}

\begin{remark}[Comparison with the result of \cite{nikolic2024infty}]
Before presenting the proof of Theorem~\ref{thm:Beale_Kato_Majda_criterion}, we compare our result to the recent work of Nikoli\'c and Winkler ~\cite{nikolic2024infty} in which they established that the $L^\infty(\Omega)$ norm of the solution of \eqref{eq:nonlinear_nonuniform} (in pressure form) with $\genk=0$ has to blow up before a gradient blow-up can occur; see~\cite[Theorem 1.1]{nikolic2024infty}.  We mention that the result of Nikoli\'c and Winkler allows for more general coefficient of the equation, e.g., $\beta \Delta u_t$ instead of $\tau c^2 \Delta u_t$. However, we do not expect this generalization to affect the present analysis, so we restrict ourselves to the relevant choice of coefficients from the modeling perspective.

To be able to compare our result to that of \cite[Theorem 1.1]{nikolic2024infty}, let us first point out that Eq.~\eqref{eq:first_eq} is a time integrated version of the equation studied in \cite{nikolic2024infty} (i.e., it is given for the pressure variable $u$ rather than the velocity potential $\psi$). The two quantities are related via the relationship \[u=\rho\psit,\] where $\rho$ is the mass density of the medium. Provided the solutions $u$ and $\psi$ are regular enough, the models are equivalent.\footnote{We refer the reader to, e.g., \cite[Section 1.9]{lions1965some}, for a rigorous proof on the equivalence of two nonlinear wave equations for two unknowns $u$ and $\psi$, where the relationship is of the form $u =\psit$.}
	
The combination of Theorem~\ref{thm:Beale_Kato_Majda_criterion} and \cite[Theorem 1.1]{nikolic2024infty} suggest that in space dimension 2, for there to be blow up in the local-in-time JMGT equation ($\genk = 0$), there has to be a simultaneous blow-up of the norms $\|\psi_t\|_{H^1(\Omega)} + \|\psi_{tt}\|_{L^2(\Omega)}$ and $\|\psi_t\|_{L^\infty(\Omega)}$ at time $T_*$. This sheds new light on the qualitative behavior near blow-up of the local JMGT model in 2D.

Furthermore, the result of Nikoli\'c and Winkler assumes considerably more regular initial data for establishing the blow-up criterion compared to their local well-posedness theory~\cite[Section 2]{nikolic2024infty}. Indeed for the blow-up characterization, they needed to assume 
\[(u_0, u_1, u_2) \in H^4(\Omega)\cap H^1_0(\Omega) \times H^2(\Omega)\cap H^1_0(\Omega) \cap H^2(\Omega)\cap H^1_0(\Omega),\]
instead of 
\[(u_0, u_1, u_2) \in H^2(\Omega)\cap H^1_0(\Omega) \times H^2(\Omega)\cap H^1_0(\Omega) \cap  H^1_0(\Omega),\]
for the local well-posedness.

In our approach, we do not need to increase the regularity of the initial data in Theorem~\ref{thm:Beale_Kato_Majda_criterion} compared to Theorem~\ref{prop::local_well_posedness}, as our strategy is to improve the estimates of Theorem~\ref{prop::local_well_posedness} assuming a condition on the boundedness of the norms in ~\eqref{Blow_up_Cond}. 
\end{remark}

We are now ready to provide the proof of Theorem~\ref{thm:Beale_Kato_Majda_criterion} and Corollary~\ref{corr:Beale_Kato_Majda} which we present for the 2D case. We comment in Remark~\ref{Remark:1D} about the changes in the 1D case.
\begin{proof}[Proof of Theorem~\ref{thm:Beale_Kato_Majda_criterion}]
	The proof follows from using Lemma~\ref{Lemma::brezis_gallouet}. 
First,  recalling \eqref{eq:first_estimate_linearized}: 
\begin{equation}\label{E_estimate_2}
	\begin{aligned}
		E(t)+ \intt \|\nabla \psitt\|_{L^2(\Om)}^2\ds + \intt \delta(\genk\Lconv \D\psitt, \D\psitt) \ds \\
		=E(0) - \intt   (c^2\D\psi - 2k_1\psit \psitt,\D\psitt)_{L^2(\Omega)}\ds.
	\end{aligned}
\end{equation}
We also have (see \eqref{eq:psit_psitt_product})
\begin{equation}\label{eq:psit_psitt_product_2}
	\begin{aligned}
	\Big|\intt  (\psit \psitt,\D\psitt)_{L^2(\Omega)} \ds\Big| &\leq  \intt \|\nabla \psit \|_{L^4(\Om)} \| \psitt \|_{L^4(\Om)} \|\nabla \psitt \|_{L^2(\Om)}\ds \\&+ \intt \|\psit\|_{L^{\infty}} \|\nabla\psitt\|^2_{L^2(\Om)}\ds. 
	\end{aligned}
\end{equation} 
	 Now,  we estimate the 
	the second term in \eqref{eq:psit_psitt_product_2} as  follows: Using Lemma~\ref{Lemma::brezis_gallouet}, we have 
	\begin{equation}\label{Term_2}
		\begin{aligned}
			\intt \|\psit\|_{L^{\infty}(\Omega)} \|\nabla\psitt\|^2_{L^2(\Om)}\ds\leq  C \intt (\|\psi_t\|_{H^1(\Omega)}\sqrt{\ln(1+\|\psi_t\|_{H^2(\Omega)})}+1) \|\nabla \psi_{tt}\|^2_{L^2(\Omega)}\ds,
		\end{aligned}
	\end{equation}
	Hence recalling \eqref{eq:Energy_def}, we can further   estimate 
	\begin{equation}\label{Estimates_BG}
		\begin{aligned}
			\intt \|\psit\|_{L^{\infty}(\Omega)} \|\nabla\psitt\|^2_{L^2(\Om)}\ds\leq  C \intt (\|\psi_t\|_{H^1(\Omega)}\sqrt{\ln(1+E(s))}+1) E(s)\ds.
		\end{aligned}
	\end{equation}
	On the other hand, we have by using Lemma \ref{lemma:ladyzhenskaya_inequality},  
	\begin{equation}\label{Estimate_Second_Term}
		\begin{aligned}  
		  \intt \|\nabla \psit \|_{L^4(\Om)}&  \| \psitt \|_{L^4(\Om)} \|\nabla \psitt \|_{L^2(\Om)}\ds \\\lesssim &\intt \|\nabla\psi_t\|_{L^2(\Omega)}^{1/2}  \|\Delta\psi_t\|_{L^2(\Omega)}^{1/2}  \|\psitt\|_{L^2(\Omega)}^{1/2} \|\nabla\psitt\|_{L^2(\Om)}^{3/2}\ds\\\lesssim &  \intt \|\psi_t\|_{H^1(\Omega)}^{1/2} \|\psitt\|_{L^2(\Omega)}^{1/2} E(s)\ds,
		\end{aligned}
	\end{equation}
In the penultimate inequality, we have used the  Poincar\'e--Friedrichs' inequality and elliptic regularity estimates by taking advantage of the fact that $\psi_t\in H^2(\Omega) \cap H_0^1(\Omega)$ and $\psitt \in H_0^1(\Omega)$. For the last inequality, we used Young's inequality  \eqref{Young_2}.
Hence, inserting \eqref{Estimates_BG} and \eqref{Estimate_Second_Term} into \eqref{E_estimate_2}, and recalling \eqref{Est_Conv} and \eqref{Linear_Term}, we obtain 
\begin{equation}\label{eqn:energy_ineq_before_Gronwall}
	\begin{aligned}
		&E(t)+ \intt \|\nabla \psitt\|_{L^2(\Om)}^2\ds + C_{\genk,\bar T}   \intt \| \genk \Lconv \D\psitt\|^2_{L^2(\Om)} \ds \\
		\lesssim&\, E(0) +\int_0^t E(s)\ds+\intt (\|\psi_t\|_{H^1(\Omega)}\sqrt{\ln(1+E(s))}+1) E(s)\ds\\
	&+\intt \|\psi_t\|_{H^1(\Omega)}^{1/2} \|\psitt\|_{L^2(\Omega)}^{1/2} E(s)\ds
	 \\
	\leq&\, C(\bar T,M_0)\left[E(0)+\intt (\sqrt{\ln(1+E(s))}+1) E(s)\ds\right],
	\end{aligned}
\end{equation}
where in the last line we have used~\eqref{Blow_up_Cond}.

To bound the energy $E(t)$ above, we follow the approach of \cite{brezis1980nonlinear}. We denote by $G(t)$ the right-hand-side of \eqref{eqn:energy_ineq_before_Gronwall}. That is 
\begin{equation}
    G(t)=C(\bar T,M_0) \left [E(0)+\intt (\sqrt{\ln(1+E(s))}+1) E(s)\right ].
\end{equation}
We have thus
\begin{equation}
	\begin{aligned}
G'(t) =&\,  C(\bar T,M_0)(\sqrt{\ln(1+E(t))}+1) E(t)\\\leq&\,  C(\bar T,M_0) (\sqrt{\ln(1+G(t))}+1) G(t),
\end{aligned}
\end{equation}
where in the last line we have used that $E(t) \leq G(t)$.
Consequently:
\begin{equation}
	\ddt [2\sqrt{\log(1+G(t))}] \leq C(\bar T,M_0).
\end{equation}
Thus, we conclude that 
\begin{equation}
	E(t)\leq G(t) \leq \exp\left[\frac{\left(C(\bar T,M_0)t + 2\sqrt{\log(1+G(0))} \right)^2}{4} \right]-1,
\end{equation}
from which it follows that the solution grows exponentially but remains bounded on the whole interval $[0,\bar T]$.
\end{proof}

\begin{remark}[On the 1D case]\label{Remark:1D}
	The proof presented above for Theorem~\ref{thm:Beale_Kato_Majda_criterion} remains valid with minimal changes in the 1D case.
 The one-dimensional case can also be optimized in view of Remark~\ref{Remark:general_d} which is given after the proof of the blow-up criterion in the 3D case established in Theorem~\ref{thm:Beale_Kato_Majda_criterion_3d}.
\end{remark}

\section{Conditional existence in the 3D case}\label{sec:cond_3D}
In this section, we state a blow-up criterion in the three-dimensional setting. The philosophy remains similar to the preceding section. Our starting point will then be a generalization of the Brezis--Gallou\"et inequality, known as Brezis--Gallou\"et-Wainger's inequality. We state in what follows a version of this inequality which will be useful to us:
\begin{lemma}[Brezis--Gallou\"et--Wainger inequality~\cite{brezis1980note}]\label{lemma:brezis_gallout_Wainger}
Let $u \in H^2(\Omega)$ with $\Omega$ a domain in $\R^3$, then
	\begin{equation}
		\|u\|_{L^\infty(\Omega)} \leq C (\|u\|_{H^\frac{3}{2}(\Omega)}\sqrt{\ln(1+\|u\|_{H^2(\Omega)})}+1),
	\end{equation}
where $C$ is a constant that depends only on $\Omega$.
\end{lemma}
\begin{proof}
The proof in \cite{brezis1980note} is done for the case $\|u\|_{L^\infty(\Omega)} \leq 1$, the general case follows from emulating \cite[Lemma 1]{brezis1980nonlinear} and changing the low-/high-frequency splitting in their proof to:
\begin{equation}\label{Eq_Fourier}
\|\hat v\|_{L^1(\R)} = \int_{\xi\leq R} (1+|\xi|)^\frac{3}{2}\hat v \frac{1}{(1+|\xi|)^\frac{3}{2}} \,\textup{d}\xi + \int_{\xi\geq R} (1+|\xi|)^2\hat v \frac{1}{(1+|\xi|)^2} \,\textup{d}\xi.
\end{equation} 
Above, $R$ is a frequency cut-off radius fixed later in the proof to be $R= \|u\|_{H^2(\Omega)}$, $\hat v$ is the Fourier transform of the extension $v= \operatorname{P} u$ of $u$ to the whole space $\R^3$.

 Using Cauchy--Schwartz inequality, we have 
 \begin{equation}
     \int_{\xi\leq R} (1+|\xi|)^\frac{3}{2}\hat v \frac{1}{(1+|\xi|)^\frac{3}{2}} \,\textup{d}\xi \lesssim \left(\int_{|\xi|\leq R}\frac{1}{(1+|\xi|)^3}\textup{d}\xi\right)^{1/2}\|v\|_{H^{3/2}}
 \end{equation}
 Passing to polar coordinates, and using the elementary inequality 
\begin{equation}\label{ineq_elem}
    \frac{1}{2}(a^r+b^r)\leq (a+b)^r\leq 2^{r}(a^r+b^r),\qquad a,\,b,\, r\geq 0,
\end{equation}
 we have 
 \begin{equation}
 \begin{aligned}
     \int_{|\xi|\leq R}\frac{1}{(1+|\xi|)^3}\textup{d}\xi\leq&\, \sigma(S^2)\int_0^R \frac{R^2}{(1+R)^3}\textup{d} R\\
     \lesssim&\,
     \int_0^R \frac{3R^2}{1+R^3}\textup{d} R\\
      \lesssim&\, \ln (1+R^3)\\
      \lesssim&\, \ln (1+R).
     \end{aligned}
 \end{equation}
 
 Above $\sigma(S^2)$ is the surface of the unit sphere in $\R^3$. In the last line, we have used the inequality~\eqref{ineq_elem} and the properties of the logarithm. Thus, the first term in~\ref{Eq_Fourier} can be estimated as:
 \[
 \int_{\xi\leq R} (1+|\xi|)^\frac{3}{2}\hat v \frac{1}{(1+|\xi|)^\frac{3}{2}} \,\textup{d}\xi \lesssim \ln^{1/2} (1+R)\|v\|_{H^{3/2}}
 \]
 The second term in \eqref{Eq_Fourier} is estimated as follows: 
 \begin{equation}
     \int_{\xi\geq R} (1+|\xi|)^2\hat v \frac{1}{(1+|\xi|)^2} \,\textup{d}\xi\lesssim \|u\|_{H^2}\frac{1}{1+R}
 \end{equation}
 Taking $R=\|u\|_{H^2}$, we obtain the desired result. 
\end{proof}

The 3D Ladyzhenskaya inequality given in Lemma~\ref{lemma:ladyzhenskaya_inequality} is unfortunately not sufficient to control the term arising from inequality~\eqref{Estimate_Second_Term} in the course of the proof of Theorem~\ref{thm:Beale_Kato_Majda_criterion}. We therefore need to establish the following lemma:
\begin{lemma}\label{lemma:L4_interp}[$L^{4}(\Omega)$ interpolation inequality]
   Let $u \in H^1(\Omega)$ with $\Omega$ a domain in $\R^3$, then
	\begin{equation}
		\|u\|_{L^4(\Omega)} \leq C \|u\|_{H^\frac{1}{2}(\Omega)}^{1/2}\|u\|_{H^1(\Omega)}^{1/2},
	\end{equation}
where $C$ is a constant that depends only on $\Omega$.
\end{lemma}
\begin{proof}
    The proof relies on the interpolation:
    \begin{equation}
    [H^{1/2}(\Omega), H^{1}(\Omega)]_{[1/2]} = H^{3/4}(\Omega),
    \end{equation}
    see, e.g., \cite[Theorem 6.4.5]{bergh2012interpolation}. 
    
    The desired inequality follows then from the fractional Sobolev embedding result~\cite[Corollary 6.7]{di2012hitchhikerʼs}.
\end{proof}

The blow-up criterion is then given as:
\begin{theorem}\label{thm:Beale_Kato_Majda_criterion_3d}
	 Let $\Omega$ be a bounded Lipschitz domain in  $\R^3$ and let Assumption~\ref{assumption1} on $\genk$ hold. Assume that $\tau>0, \, \delta\geq 0$ and let $\psi$ be a solution to \eqref{eq:nonlinear_nonuniform} and suppose that there is a time $\bar T$ such that the solution cannot be continued in $\mathcal{U}$ to $T=\bar T$. Assume that $\bar T$ is the first such time. Then 
	\begin{equation}
		\limsup_{t\uparrow \bar T} \|\psi_t\|_{H^\frac{3}{2}(\Omega)} + \|\psi_{tt}\|_{H^{\frac12}(\Omega)} = \infty.
	\end{equation}
\end{theorem}
\begin{proof}
The proof follows from emulating the steps of Theorem~\ref{thm:Beale_Kato_Majda_criterion} with a couple of minor differences. Owing to Lemma~\ref{lemma:brezis_gallout_Wainger}, inequality~\ref{Estimates_BG} becomes:
\begin{equation}\label{Estimates_BGW}
		\begin{aligned}
			\intt \|\psit\|_{L^{\infty}(\Omega)} \|\nabla\psitt\|^2_{L^2(\Om)}\ds\leq  C \intt (\|\psi_t\|_{H^\frac{3}{2}(\Omega)}\sqrt{\ln(1+E(s))}+1) E(s)\ds.
		\end{aligned}
	\end{equation}
Inequality~\eqref{Estimate_Second_Term} on the other hand becomes
\begin{equation}\label{Estimate_Second_Term_3D}
		\begin{aligned}  
		  \intt \|\nabla \psit \|_{L^4(\Om)}&  \| \psitt \|_{L^4(\Om)} \|\nabla \psitt \|_{L^2(\Om)}\ds \\\lesssim &\intt \|\nabla\psi_t\|_{H^\frac{1}{2}(\Omega)}^{1/2}  \|\Delta\psi_t\|_{L^2(\Omega)}^{1/2}  \|\psitt\|_{H^{\frac12}(\Omega)}^{1/2} \|\nabla\psitt\|_{L^2(\Om)}^{3/2}\ds\\\lesssim &  \intt \|\psi_t\|_{H^\frac{3}{2}(\Omega)}^{1/2} \|\psitt\|_{H^{\frac12}(\Omega)}^{1/2} E(s)\ds,
		\end{aligned}
	\end{equation}
where we have used Lemma~\ref{lemma:L4_interp}. In the last line we have used Young's inequality~\ref{Young_2}.

The rest of the proof follows similarly to Theorem~\ref{thm:Beale_Kato_Majda_criterion}, so we omit the rest of the details.

\end{proof}

\begin{remark}[Scaling of the blow-up criterion with the dimension $d$]\label{Remark:general_d}
    From Theorems~\ref{thm:Beale_Kato_Majda_criterion} and~\ref{thm:Beale_Kato_Majda_criterion_3d}, one can infer a scaling of the blow-up criterion with the space dimension $d$. Indeed, for $d = 2,3$, we found that the blow-up criterion is of the form:
    \[\|\psi_t\|_{H^{d/2}(\Omega)} + \|\psi_t\|_{H^{d/2-1}(\Omega)} < \infty.  \]
    This is also true for the one-dimensional case and is due to the scaling of the inequalities on which the analysis is based: The Brezis--Gallou\"et--Wainger inequality~\cite{brezis1980note} and variations of the Ladyzhenskaya inequality which can be obtained via interpolation of Sobolev spaces and their embeddings. We note that, when comparing this result to~\cite{nikolic2024infty}, $\|\psi_t\|_{H^{d/2}(\Omega)}$ has the same scaling as the $L^\infty(\Omega)$ norm, in the sense that we have the embedding $H^{\frac{d}{2}+\epsilon}(\Omega)\hookrightarrow L^\infty(\Omega)$ for all $\epsilon>0$, but $H^{d/2}(\Omega)\not\hookrightarrow L^\infty(\Omega)$.
\end{remark}

\section{Conclusions and outlook}
In this work, we established  a more informative local well-posedness of a fractionally damped nonlinear JMGT equation and provided a lower bound on the final time of existence as a function of the size of the initial data. 
Our analysis addressed key aspects of the behavior of the solution and highlighted the impact of the size of the initial condition on the lifespan of the solution. 
We have further investigated the conditional existence of the solutions and compared our results to the recent work of \cite{nikolic2024infty}.

One of the important future research directions is to gain a deeper understanding and characterize the qualitative behavior of blow-up in the nonlinear JMGT equation. 
It is indeed clear that the qualitative behavior near the blow-up time is significantly influenced by the space dimension of the problem.
Thus, of great interest is the construction of solutions exhibiting blow-up in the different space dimensions $d = 1,\,2,\,3$, which will help understand the blow-up mechanisms of the JMGT equation.

Additionally, it is critical to establish sufficient conditions on the fractional kernel $\genk$ to ensure the existence of global-in-time solutions. In the current study, we have based our analysis solely on the minimal  Assumption~\ref{assumption1}, which serves as a key ingredient for the local analysis. However, by imposing stronger assumptions on the kernel, it is reasonable to expect that one can obtain the global existence of solutions (similarly to what is known in the case $\genk$ is the Dirac delta distribution). Investigating the asymptotic behavior of the solution when  $\delta>0$ is very valuable (for $\delta=0$, it was proved in \cite{Pellicer:2021aa} that the operator associated with the linear part of the JMGT equation has eigenvalues with nonnegative real parts). To investigate this, it is expected that strong assumptions on the convolution kernel are needed. Identifying a minimal set of assumptions on $\genk$ will contribute not only to the understanding of the fractional JMGT models but also, more generally, to the field of fractional and wave equations with memory.

\section*{Acknowledgments}
The work of Mostafa Meliani was performed while he was a PhD candidate at the Department of Mathematics of Radboud University, The Netherlands.  The authors are grateful to Dr.
Vanja Nikoli\'c (Radboud University) for her careful reading of the manuscript and suggestions for improvements.

\end{document}